\documentclass[journal]{IEEEtran}
\usepackage{amsmath,amsfonts,amssymb,euscript, graphicx, 
epsfig,
enumerate,float,afterpage, subfigure, ifthen}%

\usepackage{url}
\usepackage{hyperref}

\newtheorem{thm}{Theorem}

\newtheorem{lem}{Lemma}

\newcommand{\norm}[1]{||{#1}||}

\newcommand{\script}[1]{{{\cal{#1} }}}

\begin{document}

\title
  {Sharing Information Without Regret in Managed Stochastic Games}
\author{Michael J. Neely\\University of Southern California\\\url{http://www-bcf.usc.edu/~mjneely}
\thanks{The author is with the  Electrical Engineering department at the University
of Southern California, Los Angeles, CA.} 
\thanks{This work is supported in part  by the NSF Career grant CCF-0747525.}
}

\markboth{}{Neely}

\maketitle

\begin{abstract}   
This paper considers information sharing in a multi-player repeated game.  Every round, each player observes a subset of components of a random vector and then takes a control action.  The utility earned by each player depends on the full random vector and on the actions of others.  An example is a game where different rewards are placed over multiple locations, each player only knows the rewards in a subset of the locations, and players compete to collect the 
rewards. Sharing information can help others, but can also increase competition for desirable locations. Standard Nash equilibrium and correlated equilibrium concepts are inadequate in this scenario. Instead, this paper develops an algorithm where, every round, all players pass their information and intended actions to a game manager.  The manager provides suggested actions for each player that, if taken, maximize a concave function of average utilities subject to the constraint that each player gets an average utility no worse than it would get without sharing.  The algorithm acts online using information given at each round and does not require a specific model of random events or player actions. Thus, the analytical results of this paper apply in non-ergodic situations with any sequence of actions taken by human players. 
\end{abstract} 

\section{Introduction}

This paper considers a stochastic game where each player has incomplete information.  
A central issue is whether or not this information should be shared.   Indeed, players with special access to desirable information may prefer to keep this information private. The goal of this paper is to design an efficient collaborative strategy that allows players to share information  without sacrificing their own interests. 

The general game structure is as follows:  There are $N$ players that repeatedly play a game over 
a sequence of  rounds $t \in \{0, 1, 2, \ldots\}$.  On each round $t$, there is a random event vector $\omega(t)=(\omega_1(t), \ldots, \omega_M(t))$ that describes characteristics of the game for that round. The value $M$ is a positive integer that can be different from $N$.   Each player can observe a  \emph{portion} of the components of the $\omega(t)$ vector.  Specifically, for each $i \in \{1, \ldots, N\}$,  define $\script{S}_i$ as the subset of indices in $\{1, \ldots, M\}$ that are observable by player $i$.   Every round $t$, each player $i$ observes its components of $\omega(t)$ and then chooses an action $\alpha_i(t)$ based on this (incomplete) information. Define $\alpha(t) = (\alpha_1(t), \ldots, \alpha_N(t))$ as the joint action vector.  The resulting round-$t$ payoff for player $i$ is $u_i(t)$, also called the \emph{utility}.  The utility $u_i(t)$ is a general function of $\alpha(t)$ and $\omega(t)$: 
\[ u_i(t) = \hat{u}_i(\alpha(t), \omega(t)) \: \: \forall i \in \{1, \ldots, N\} \]
Each player wants to earn a large time average utility $\overline{u}_i$: 
\[ \overline{u}_i = \lim_{t\rightarrow\infty} \frac{1}{t}\sum_{\tau=0}^{t-1} u_i(\tau) \]

Pooling information about the $\omega(t)$ vector and making a \emph{team decision} can improve the sum utility.  However, individual players may want to keep their information private to increase their own utility. 

\subsection{Example game structure} \label{section:location-reward} 

\begin{figure}[htbp]
   \centering
   \includegraphics[width=2.5in]{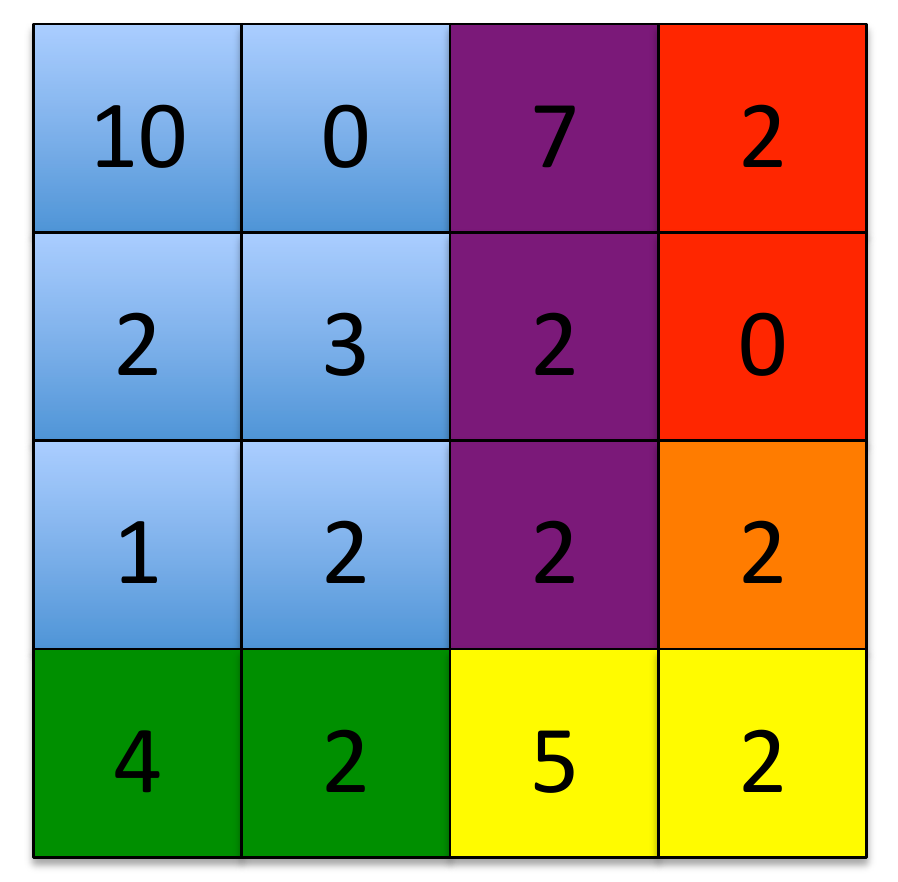} 
   \caption{An illustration of a 3-player location-reward game. Locations known only to player $i$ are blue, yellow, red for $i \in \{1,2,3\}$, respectively.  Locations known to both players 1 and 2 are green, players 2 and 3 are orange, and players 3 and 1 are purple.}
   \label{fig:games-grid-color}
\end{figure}

Consider a square region that is partitioned into $M$ disjoint sub-regions, called \emph{locations} (see Fig. \ref{fig:games-grid-color} with $M=16$ locations). 
Every round, a random reward $\omega_m(t)$ appears in each location $m \in \{1, \ldots, M\}$.  
Let $\omega(t) = (\omega_1(t), \ldots, \omega_M(t))$ be the vector of current rewards.  
For example, $\omega(t)$ might be a random vector that is independent and identically distributed (i.i.d.) over rounds $t \in \{0, 1, 2, \ldots\}$ with some arbitrary  joint probability distribution. 
Suppose there are 3 players: 
\begin{itemize} 
\item Player 1 knows  rewards for the blue, purple, and green squares. 
\item Player 2 knows rewards for the yellow, green, and orange squares. 
\item Player 3 knows rewards for the red, orange, and purple squares. 
\end{itemize} 

Every round, each player chooses a single location where it competes for the current reward. 
Specifically, 
for each $i \in \{1, 2, 3\}$, let $\script{A}_i$ be a subset of $\{1, \ldots, M\}$ that represents the set of locations player $i$ is allowed to choose from, called the \emph{action set} for player $i$. The sets $\script{A}_i$ and $\script{S}_i$ can be different, so that a player might choose a location in which she does not know the reward. 
Let $\alpha_i(t)$ be the location in $\script{A}_i$ chosen by player $i$ on round $t$. If a player is the only one to choose a certain location $m$, she earns the full reward $\omega_m(t)$. Else, the reward is split evenly amongst all players who choose that location. Specifically, for each $m \in \{1, \ldots, M\}$ 
define $K_m(t)$ as the number of players who choose location $m$ on round $t$.  The resulting utility for player $i\in \{1, 2, 3\}$ is: 
\[ u_i(t) = \frac{\omega_{\alpha_i(t)}(t)}{K_{\alpha_i(t)}(t)} \]
This utility is indeed a function of the vectors $\alpha(t)$ and $\omega(t)$: 
\begin{equation} \label{eq:example-utility} 
 \hat{u}_i(\alpha, \omega) = \frac{\omega_{\alpha_i}}{\sum_{n=1}^3 1\{\alpha_n=\alpha_i\}} 
 \end{equation} 
where $1\{\alpha_n= \alpha_i\}$ is  an indicator function that is $1$ if $\alpha_n=\alpha_i$, and $0$ else. The denominator in \eqref{eq:example-utility} 
is always nonzero since $1\{\alpha_i = \alpha_i\}=1$ for all $i \in \{1, 2, 3\}$. 

For the scenario illustrated in Fig. \ref{fig:games-grid-color}, player 1 is the only one to see the highly desirable reward of 10 that is currently in the top left location.  Player 1 might want to keep this information private to reduce the chance of other players competing for the same location. 

\subsection{Prior work on repeated games} 

 Adaptive methods that converge to a \emph{correlated equilibrium} for repeated play of 
 static games are developed in  \cite{foster-vohra-games-CE}\cite{hart-correlated-games-CE}\cite{fudenberg-games-CE}. Correlated equilibrium in stochastic games is
 considered in \cite{cor-eq-stochastic-games}\cite{repeated-games-arxiv}.  The formulation of the current paper is the most similar to \cite{repeated-games-arxiv}, where a game manager helps to achieve 
 correlated and coarse correlated equilibrium in repeated stochastic games.  However, the prior work \cite{foster-vohra-games-CE}\cite{hart-correlated-games-CE}\cite{fudenberg-games-CE}\cite{cor-eq-stochastic-games}\cite{repeated-games-arxiv} does not consider the problem of information sharing, and so the notions of equilibrium they study do not directly apply in the current context.  Further, the work on stochastic games in \cite{cor-eq-stochastic-games}\cite{repeated-games-arxiv} considers an ergodic regime, while the current paper considers arbitrary sample paths that are possibly non-ergodic. 

\subsection{Inadequacy of standard equilibrium definitions} 

Standard definitions of Nash equilibrium \cite{nash-games}\cite{nash-n-person-games}, correlated equilibrium \cite{aumann-correlated-eq1}\cite{aumann-correlated-eq2}, and coarse correlated equilibrium  \cite{CCE} are inadequate in the scenario of this paper. That is because such equilibrium definitions require the utility of each player $i$ to be at least as large as it would be if player $i$ individually deviated from the intended strategy \emph{while all other players continue to 
use the intended strategy}.  However, if player $i$ deviates by choosing 
not to share information, the intended strategies of others may no longer be possible  because they might rely on this  information.  

In principle, one could circumvent this difficulty by forcing the repeated game structure to look like a 1-shot game for which standard notions of equilibrium exist.  For example, this could be done by allowing each player to make a binary decision at the start that determines whether or not she will share information.  Her remaining decisions can be viewed as an 
element of a strategy space defined over infinite sequences of actions.   This approach is taken in \cite{cor-eq-stochastic-games}
by using the concept of \emph{infinitely punishing deviant behavior} (also see discussions in 
\cite{game-theory-book}).  There, players are assigned strategies that require them to maximally punish any non-conformist by taking actions (for all time) that are solely designed to yield poor utility 
for the non-conformist.   This can lead to an equilibrium because such punishment never occurs 
(since all players conform out of fear), and so (mathematically) players do not object to having 
an unused requirement to punish others as part of their decision strategy.  
However, this approach does not necessarily capture realistic behavior. Human players will \emph{not}  spend the rest of their lives punishing a non-conformist.  Rather, human players will adapt their behavior to emerging conditions.  The mathematical threat of infinite punishment is seen to be a sham that lacks power to realistically influence behavior. 

A modified definition of \emph{subgame perfect equilibrium} is often used as an attempt to make punishment threats credible \cite{game-theory-book}.   It requires equilibrium-type conditions to be met even for unused punishment modes of a strategy.  Such conditions can often be met by using finite-length punishment modes.  However, in many repeated games, such modes can be appended to almost any strategy to endow that strategy with the subgame perfect equilibrium property (see theorems on repeated games with time average utility metrics in \cite{game-theory-book}).  

Overall, while prior notions of equilibrium for repeated games exist and have well defined mathematical properties, they can be complex and difficult for humans to  interpret.  
Arguably, a human player wants a more direct comparison of the suggested strategy with some other reasonable course of action.  This paper uses a much simpler ``no regret'' guarantee that compares the suggestions to a single 
baseline strategy, rather than to all possible strategies. The baseline strategy is defined by any alternative sequence of actions that a player wants to consider.   The challenge  is to dynamically make suggestions as the game is being played. The suggestions must meet the desired performance for arbitrary sample paths, and therefore must adapt as new events emerge.

\subsection{The game manager} 

This paper deviates from the standard equilibrium approach by assuming the existence of a \emph{game manager} to which players pass their information.  Further, every round 
it is assumed that each player $i \in \{1, \ldots, N\}$ has a \emph{baseline decision} $b_i(t) \in \script{A}_i$. 
The baseline decision $b_i(t)$ is any decision the human player would choose if she were operating on her own without a manager and without information sharing.   The manager takes the given information every round $t$ and produces \emph{suggested actions} for each player.  The suggested actions must have the property that, if every player uses the suggestions, a concave function of average utilities across players is maximized subject to the constraint that each player receives an average utility at least as large as the average 
value it would earn if all players used their baseline strategies.  The running average utility gains between the suggested and baseline strategies can be given to each player on each round.   The understanding is that players have more incentive to take suggested actions if they see the 
gains of doing so. 
  This approach does not rely on punishment modes in a complex strategy space.  

This formulation is interesting because it defines a specific optimization problem that uses the (possibly human-generated) sequence of baseline strategies as part of the optimization.   The solution approach uses Lyapunov optimization theory in this new context.   Lyapunov optimization is known to have \emph{universal scheduling} properties that provide analytical guarantees for arbitrary sequences \cite{sno-text}.  Those properties are used in the game theory context of this paper to provide a simple online algorithm for making manager decisions as the game is played.  This paper shows that the resulting algorithm  provides analytical guarantees for arbitrary baseline sequences and arbitrary event sequences, including sequences with no probabilistic description. 

\section{Examples} 

To gain intuition, this section provides examples of a 2-player location-reward game with only two locations. 
The random event vectors
$\omega(t)=(\omega_1(t), \omega_2(t))$ are assumed to be independent and identically distributed (i.i.d.) over rounds $t \in \{0, 1,2, \ldots\}$.  This allows exact computation of average utility associated with 
different strategies.  
The examples are designed to show that willingness to share information can depend on the statistical distribution of rewards and also on the constraint sets of individual players.   
For simplicity, the examples in this section assume the full joint probability distribution of $(\omega_1(t), \omega_2(t))$ is known by both players. The general model of Section \ref{section:problem} treats a more complex scenario where 
the event vectors $\omega(t)$ are arbitrary sequences, possibly non-ergodic sequences with no known probabilistic structure.

\begin{figure}[htbp]
   \centering
   \includegraphics[width=1.5in]{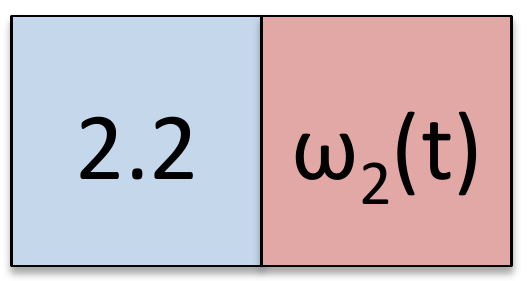} 
   \caption{A location-reward game with two locations and two players. Player 1 knows the reward in the left location, but this reward is always 2.2.  Player 2 knows the random reward $\omega_2(t)$ in the right location.}
   \label{fig:two-player-locations}
\end{figure}

\subsection{Example 1: Beneficially withholding information} \label{section:example1}

Consider a location-reward game (as described in Section \ref{section:location-reward}) with 
two players and two locations, as shown in Fig. \ref{fig:two-player-locations}.  Player 1 knows the reward $\omega_1(t)$ associated with location 1 and player 2 knows the reward $\omega_2(t)$ associated with location 2.  Every round, player 1 can choose from either of the two locations.  However, suppose that player 2 is restricted to only choosing location 2.   The reward probabilities are: 
\begin{eqnarray*} 
\omega_1(t) &=& 2.2 \: \: \mbox{ with probability 1} \\
\omega_2(t) &=&  \left\{ \begin{array}{ll}
10 &\mbox{ with probability $1/5$} \\
2  & \mbox{ with probability $4/5$} 
\end{array}
\right. 
\end{eqnarray*}
The vectors $\omega(t) = (\omega_1(t), \omega_2(t))$ are i.i.d. over rounds $t$, and the above probabilities are known to both players. Player 2 is the only one with knowledge of the actual realization of the time-varying reward in location 2. 

\subsubsection{Without sharing information} Suppose player 2 does not share its knowledge. If player 1 chooses location 2, its expected utility is $5(1/5) + 1(4/5) = 1.8$, which is strictly less than the utility of 2.2 it would achieve by choosing location 1.  Hence, the optimal strategy for player 1 is to always choose location 1.  Assuming that player 1 uses this optimal strategy, the resulting average utilities are: 
\begin{eqnarray}
\overline{u}_1 &=& 2.2 \label{eq:earned1} \\
\overline{u}_2 &=& 10(1/5) + 2(4/5) = 3.6 \label{eq:earned2} 
\end{eqnarray}

\subsubsection{Sharing information} Suppose player 2 chooses to always 
divulge the value of $\omega_2(t)$ before either player makes a decision.  In this case, it is optimal for player 1 to choose location 2 whenever $\omega_2(t)=10$, and to choose location 1 otherwise.  The resulting utilities under this strategy 
are: 
\begin{eqnarray*}
\overline{u}_1 &=&  2.2(4/5) + 5(1/5) = 2.76\\
\overline{u}_2 &=& 2(4/5) + 5(1/5) = 2.6
\end{eqnarray*}
In comparison to the utilities of \eqref{eq:earned1}-\eqref{eq:earned2}, it is clear that player 1 increases her utility by taking advantage of the shared information.  However, player 2 \emph{reduces} her utility because she now competes for the desirable reward of 10 when that reward appears.  Thus, in this example, player 2 has no incentive to share information.  Player 2 prefers to keep her information private. 

\subsection{Example 2: Beneficially sharing information}  \label{section:example2} 

Consider the same two-player location-reward game as the previous subsection.  
The only difference is that the reward probabilities are different: 
\begin{eqnarray} 
\omega_1(t) &=& 2.2 \: \: \mbox{ with probability 1}  \label{eq:example1} \\
\omega_2(t) &=& \left\{ \begin{array}{ll}
10 &\mbox{ with probability $1/2$} \\
2  & \mbox{ with probability $1/2$}  
\end{array}
\right. \label{eq:example2} 
\end{eqnarray}

\subsubsection{Without sharing information} 

Without sharing information, it is optimal for player 1 to always choose location 2.  The resulting utilities under this strategy are: 
\begin{eqnarray*}
\overline{u}_1 = \overline{u}_2 = 5(1/2) + 1(1/2) = 3 
\end{eqnarray*}

\subsubsection{Sharing information} 

If player 2 shares $\omega_2(t)$ with player 1 every round $t$, then the optimal strategy for player 1 is to choose location 2 whenever $\omega_2(t)=10$, and choose location 1 otherwise. The resulting utilities under this optimal strategy are: 
\begin{eqnarray*}
\overline{u}_1 &=& 2.2(1/2) + 5(1/2) = 3.6 \\
\overline{u}_2 &=& 2(1/2) + 5(1/2) = 3.5
\end{eqnarray*}
In this example, sharing information allows \emph{both} players to increase their average utility.  Player 2 wins by revealing the value of $\omega_2(t)$ because it discourages player 1 from competing 
when the reward is small.    

The examples in Sections \ref{section:example1} and \ref{section:example2} show that changing the probability distribution of random events can impact the optimality or sub-optimality of sharing information.  
The problem treated in this paper is even more challenging 
because the probabilities are unknown and can possibly change. 

\subsection{Example 3: Unrestricted actions} 

Now consider the same example of the previous subsection, with the same reward probabilities as \eqref{eq:example1}-\eqref{eq:example2}. The only difference is that now both players 1 and 2 are free to select either of the two locations. 

\subsubsection{Without sharing information} 

Without knowledge of $\omega_2(t)$, it is optimal for player 1 to always choose location 2. Indeed, the \emph{smallest} expected reward it can earn by doing this is computed by assuming it must always share its rewards in location 2, and this leads to a value of $5/2 + 1/2 = 3 > 2.2$. If player 2 assumes that player 1 uses its optimal policy, then its best strategy is to choose location 2 whenever $\omega_2(t)=10$, and to choose location 1 if $\omega_2(t)=2$.   Under these optimized strategies, the utilities are: 
 \begin{eqnarray*}
\overline{u}_1 &=& 5(1/2) + 2(1/2) = 3.5 \\
\overline{u}_2 &=& 5(1/2) + 2.2(1/2) = 3.6
\end{eqnarray*}

\subsubsection{Sharing information} 

Suppose player 2 shares $\omega_2(t)$ with player 1 every round $t$.  In this case, it is optimal for both players to choose location 2 whenever $\omega_2(t)=10$.  However, the optimal decisions for players 1 and 2 are unclear when $\omega_2(t)=2$, since optimality of each player depends on the policy implemented by the other.  However, the \emph{highest} utility that player 2 can achieve is $3.6$ (which assumes it always gets exclusive access to the 2.2 reward in location 1 when $\omega_2(t)=2$).  Thus, at best, player 2 can only achieve the same utility by sharing its information, but more likely stands to \emph{loose} utility when this information is shared.  Therefore, in this example, a rational and self-interested player 2 would not want to share information. 

\section{The general model} \label{section:problem} 

Fix $N$ as an integer larger than 1. 
Suppose there are $N$ players that play a game over rounds $t\in\{0, 1, 2, \ldots\}$. 
Define $\script{N} = \{1, \ldots, N\}$ as the set of players. 
Fix $M$ as a positive integer (possibly different from $N$) and suppose that $\omega(t) = (\omega_1(t), \omega_2(t), \ldots, \omega_M(t))$ is a sequence of event vectors over rounds $t \in \{0, 1, 2, \ldots\}$.  Assume that $\omega(t)$ takes values in some abstract (possibly infinite) set $\Omega$ on each round $t$. 
The vector sequence $\omega(t)$ is otherwise 
arbitrary and can have arbitrary correlations over entries and over time.  A probabilistic description of $\omega(t)$ is not necessarily known to the players, and such a probabilistic description may not even exist. 
For each $i \in \script{N}$, define $\script{S}_i$ as the subset of  $\{1, \ldots, M\}$ associated with components of  $\omega(t)$ that player $i$ can observe at the beginning of each round.  That is, player $i$ knows $\omega_j(t)$ for all $j \in \script{S}_i$ and for all $t$.  It is assumed that $\cup_{i=1}^N \script{S}_i = \{1, \ldots, M\}$, so that the combined knowledge of all players gives the full $\omega(t)$ vector. 

Every round $t$, each player $i\in \script{N}$ observes its components of $\omega(t)$ and 
chooses an \emph{action} $\alpha_i(t)$ as an element in some abstract (possibly infinite) 
action set $\script{A}_i$.  Let $\alpha(t) = (\alpha_1(t), \ldots, \alpha_N(t))$ be the action vector. The resulting utility $u_i(t)$ earned by player $i$ on round $t$ is a general function of $\alpha(t)$ and $\omega(t)$: 
\[ u_i(t) = \hat{u}_i(\alpha(t), \omega(t)) \: \: \forall i \in \script{N} \]

\subsection{Assumptions} 

The functions $\hat{u}_i(\alpha, \omega)$ are assumed to be non-negative and upper-bounded.  Specifically, for each $i \in \script{N}$, assume there is a maximum utility value $u_i^{max}<\infty$ such that: 
\[ 0 \leq \hat{u}_i(\alpha, \omega) \leq u_i^{max} \: \: \: \: \forall (\alpha, \omega)  \in \script{A}_1\times \cdots \times \script{A}_N \times \Omega \]
For simplicity, further assume the utility functions are such that for every 
$\omega \in \Omega$, the problem of choosing an action vector $\alpha = (\alpha_1, \ldots, \alpha_N)$ to maximize a weighted sum of utilities has a well defined (possibly non-unique) 
maximizing solution $\alpha^*$.  Specifically, the following problem has a well defined maximum (so that the supremum objective function value is achievable) 
 for all possible real numbers $\beta_i$: 
\begin{eqnarray}
\mbox{Maximize:} & \sum_{i=1}^N \beta_i \hat{u}_i(\alpha, \omega) \label{eq:simplicity1}  \\
\mbox{Subject to:} & \alpha_i \in \script{A}_i \: \: \forall i \in \script{N} \label{eq:simplicity2} 
\end{eqnarray}
This is a mild assumption that holds in most practical cases.\footnote{Assuming existence of a maximizer $\alpha^*$ for the problem \eqref{eq:simplicity1}-\eqref{eq:simplicity2} simplifies exposition but is not crucial to the analysis. This assumption can be avoided  by using the \emph{$C$-additive approximation} theory in \cite{sno-text}.}   For example, 
\eqref{eq:simplicity1}-\eqref{eq:simplicity2}  is guaranteed to have a maximizing solution $\alpha^* = (\alpha_1^*, \ldots, \alpha_N^*)$ when all sets $\script{A}_i$ are finite.  It is also guaranteed to have a well defined maximizer when the sets $\script{A}_i$ are infinite but are 
compact subsets of a finite-dimensional vector space,  and when the utility functions $\hat{u}_i(\alpha, \omega)$ are continuous in $\alpha$ for all $\omega \in \Omega$.  

The utility functions are otherwise arbitrary.  In particular, they are not required to have 
convexity or concavity properties. 

\subsection{Baseline actions and the game manager} 

Every round $t$, each player $i$ observes $\omega_j(t)$ for all $j \in \script{S}_i$ and then 
makes a \emph{baseline decision} $b_i(t) \in \script{A}_i$.  The resulting sequence $\{b_i(t)\}_{t=0}^{\infty}$ can be arbitrary and has no assumed structure. 
However, the understanding is that the baseline decision on round $t$ is an action that player $i$ would want to take if
it did not have access to shared information or to suggestions of a game manager.  It can be based on 
the player $i$ observations of current and past events. 

At the beginning of each round $t$, all players $i \in \script{N}$ privately 
upload their observations $\omega_j(t)$  (for all $j \in \script{S}_i$) and their baseline decisions $b_i(t)$ to the game manager.  Thus, on round $t$, the manager knows the full $\omega(t)$ vector, all of the baseline decisions $b_i(t)$,
and the complete history of past events.  It is assumed that the manager has only \emph{causal knowledge}, and hence it does not know the future values $\omega(\tau)$ and $b_i(\tau)$ for $\tau>t$.  
The manager uses its information on round $t$ to compute \emph{suggested actions} $\tilde{\alpha}_i(t) \in \script{A}_i$ that it delivers to each individual player.   Define $\tilde{\alpha}(t) = (\tilde{\alpha}_1(t), \ldots, \tilde{\alpha}_N(t))$ and $b(t) = (b_1(t), \ldots, b_N(t))$. 

\subsection{An (overly?) ambitious optimization problem} 

Define $u_i(t)$ and $x_i(t)$ for each  $i \in \script{N}$ and each $t \in \{0, 1, 2, \ldots\}$  by: 
\begin{eqnarray*}
u_i(t) &=& \hat{u}_i(\tilde{\alpha}(t), \omega(t)) \\
x_i(t) &=& \hat{u}_i(b(t), \omega(t)) 
\end{eqnarray*}
The value $u_i(t)$ is the utility earned by player $i$ on round $t$ if all players choose the suggested actions, while $x_i(t)$ is the corresponding utility if all players choose their baseline decisions. 
Define time averages for all $t>0$ by: 
\begin{eqnarray*}
\overline{u}_i(t) &=& \frac{1}{t}\sum_{\tau=0}^{t-1} \hat{u}_i(\tilde{\alpha}(t), \omega(t)) \\
\overline{x}_i(t) &=& \frac{1}{t}\sum_{\tau=0}^{t-1} \hat{u}_i(b(t), \omega(t)) 
\end{eqnarray*} 
It is useful to introduce an ambitious optimization problem that will be modified later. 
Subject to the emerging $\omega(t)$ and $b(t)$ sequences, the goal is for the manager to make suggestions $\tilde{\alpha}(t)$ that solve the following: 
\begin{eqnarray}
\mbox{Maximize:} & \liminf_{t\rightarrow\infty} \phi(\overline{u}_1(t), \ldots, \overline{u}_N(t)) \label{eq:p10} \\
\mbox{Subject to:} & \liminf_{t\rightarrow\infty} [\overline{u}_i(t) - \overline{x}_i(t)] \geq 0 \: \: \forall i \in \script{N} \label{eq:p20} \\
& \tilde{\alpha}_i(t) \in \script{A}_i \: \: \forall i \in \script{N}, \forall t \in \{0, 1,2, \ldots\} \label{eq:p30} 
\end{eqnarray}
where $\phi(u_1, \ldots, u_N)$ is a continuous and concave  function defined over the hyper-rectangle of all 
$(u_1, \ldots, u_N) \in \prod [0, u_i^{max}]$.  The constraint \eqref{eq:p20} ensures the time average utility player $i$ receives if all players use the suggestions  of the manager is at least as good as the utility it would receive  
if all players used their baseline strategies.  

The \emph{definition} of optimality for  \eqref{eq:p10}-\eqref{eq:p30} requires a more careful treatment, and this issue is discussed  more precisely in the next subsection. 
In particular, the \emph{causality constraint} of the game manager does not explicitly appear anywhere in 
\eqref{eq:p10}-\eqref{eq:p30}. 
Regardless, it can be shown that the above problem is always \emph{feasible}, so that it is always possible 
to satisfy constraints \eqref{eq:p20}-\eqref{eq:p30} via a simple causal algorithm for game manager decisions: Consider the trivial decisions  $(\tilde{\alpha}_1(t), \ldots, \tilde{\alpha}_N(t)) = (b_1(t), \ldots, b_N(t))$ for all $t\in \{0, 1,2, \ldots\}$.  This means that the game manager suggests nothing more than the baseline actions for all rounds $t$.  These suggestions are implementable in a causal manner because they only use the baseline decisions given to the game manager at the beginning of each round. Since $b_i(t) \in \script{A}_i$ for all $i$ and all $t$, one has  $\tilde{\alpha}_i(t) = b_i(t) \in \script{A}_i$ for all $i$, and so  the constraints  \eqref{eq:p30} are satisfied. Further, this trivial suggestion strategy immediately  implies 
$u_i(t) = x_i(t)$ for all $t$, and so constraint \eqref{eq:p20} is trivially satisfied. 

Thus, it is \emph{always possible} for a game manager to achieve the constraints \eqref{eq:p20}-\eqref{eq:p30}. 
The constraints \eqref{eq:p20} provide rational players an incentive to use the suggested actions, since it ensures the resulting time average utilities are at least as good as those of the baseline decisions.  Thus, it is assumed throughout this paper  
that all players choose the suggestions of the manager, so that $\tilde{\alpha}_i(t) = \alpha_i(t)$ for all $i$ and all $t$.  In particular, one has for all rounds $t$:  
\begin{eqnarray}
u_i(t) &=& \hat{u}_i(\alpha(t), \omega(t)) \label{eq:ui} \\
x_i(t) &=& \hat{u}_i(b(t), \omega(t)) \label{eq:xi} 
\end{eqnarray}

Define: 
\begin{eqnarray*}
u(t) &=& (u_1(t), \ldots, u_N(t)) \\
\overline{u}(t) &=& (\overline{u}_1(t), \ldots, \overline{u}_N(t)) 
\end{eqnarray*}
Under the assumption $\tilde{\alpha}_i(t)=\alpha_i(t)$ for all $i$ and all $t$, the problem becomes:
\begin{eqnarray}
\mbox{Maximize:} & \liminf_{t\rightarrow\infty} \phi(\overline{u}(t)) \label{eq:p1} \\
\mbox{Subject to:} & \liminf_{t\rightarrow\infty} [\overline{u}_i(t) - \overline{x}_i(t)] \geq 0 \: \: \forall i \in \script{N} \label{eq:p2} \\
& \alpha_i(t) \in \script{A}_i \: \: \forall i \in \script{N}, \forall t \in \{0, 1,2, \ldots\} \label{eq:p3} 
\end{eqnarray} 
where $u_i(t)$ and $x_i(t)$ are defined in \eqref{eq:ui} and \eqref{eq:xi}. 
The problem \eqref{eq:p1}-\eqref{eq:p3} shall be referred to as the \emph{infinite future knowledge optimization problem}.

\subsection{A modified objective} 

The infinite sequences $\omega(t)$ and $b(t)$ for $t \in \{0, 1,2, \ldots\}$ are arbitrary.  In practice, these sequences might depend on previous suggestions of the game manager.  However, the problem \eqref{eq:p1}-\eqref{eq:p3} does not specify how future values of $\omega(\tau)$ and $b(\tau)$ depend on control decisions. 
  Thus, to understand optimal utility  in \eqref{eq:p1}-\eqref{eq:p3}, it useful to view
 $\{\omega(t)\}_{t=0}^{\infty}$ and $\{b(t)\}_{t=0}^{\infty}$ as  arbitrary sequences 
  that are defined at the start of the game but with values that are only sequentially revealed as the game progresses.  In this way, future values of $\omega(\tau)$ and $b(\tau)$ are not influenced by the emerging decisions of the manager.

  With this structure, optimality
 of \eqref{eq:p1}-\eqref{eq:p3}  
   is defined over all possible action sequences $\{\alpha(t)\}_{t=0}^{\infty}$ that provide utilities with respect to the given $\{\omega(t)\}_{t=0}^{\infty}$ and $\{b(t)\}_{t=0}^{\infty}$ sequences. 
   In principle, the supremum time average utility in \eqref{eq:p1}  can 
   be computed offline based on non-causal 
   knowledge of the full sequences $\{\omega(t)\}_{t=0}^{\infty}$ and $\{b(t)\}_{t=0}^{\infty}$.
 Since optimality is defined in terms of full knowledge of the future, the problem   \eqref{eq:p1}-\eqref{eq:p3}  is called the \emph{infinite future knowledge optimization problem}. 
It is not clear if the supremum objective function value for this problem 
can be achieved by a practical algorithm that makes causal decisions.  The stochastic optimization theory in \cite{sno-text} shows that the supremum \emph{can} be achieved in a causal manner in the special case when the sequences $\{\omega(t)\}_{t=0}^{\infty}$ and $\{b(t)\}_{t=0}^{\infty}$ have an \emph{ergodic} structure.\footnote{Specifically, \cite{sno-text} shows optimality can be achieved (arbitrarily closely) in the case when the random event process is modulated by a finite state irreducible (possibly periodic) discrete time Markov chain.}  That is because, in the ergodic case,  the problem \eqref{eq:p1}-\eqref{eq:p3} fits into the general framework of \cite{sno-text}. 

However, this paper considers problems with possibly non-ergodic input sequences $\{\omega(t)\}_{t=0}^{\infty}$ and $\{b(t)\}_{t=0}^{\infty}$.  In this context, it is not clear if 
the supremum in \eqref{eq:p1}-\eqref{eq:p3} can be causally achieved. Remarkably, this paper shows that a modified objective, defined in terms of a finite but arbitrarily large window of time in which the future is known, \emph{can in fact be causally achieved} (to within any arbitrarily small but positive error).  This is done using the \emph{$T$-slot lookahead utility} developed in \cite{sno-text}:  Fix $T$ as a positive integer and partition the rounds $t \in \{0, 1,2, \ldots\}$ into successive frames of size $T$, so that frame $k$
consists of rounds $\{kT, \ldots, (k+1)T-1\}$ (for each $k \in \{0, 1,2, \ldots\}$).  For each round $k$ and for given 
realizations of 
$\{\omega(t)\}_{t=0}^{\infty}$ and $\{b(t)\}_{t=0}^{\infty}$, define 
$\psi_T[k]$ as the supremum objective value in the following optimization problem, optimized over all choices of the decision vector $\alpha(t) = (\alpha_1(t), \ldots, \alpha_N(t))$: 
\begin{eqnarray}
\mbox{Max:} & \phi\left(\gamma_1, \ldots, \gamma_N\right) \label{eq:tslot-1} \\
\mbox{Subj to:} & \gamma_i = \frac{1}{T}\sum_{\tau=kT}^{(k+1)T-1} u_i\left(\alpha(\tau), \omega(\tau)\right) \: \: \forall i \in \script{N} \label{eq:tslot-2} \\
& \gamma_i \geq \frac{1}{T}\sum_{\tau=kT}^{(k+1)T-1}u_i(b(\tau), \omega(\tau)) \: \: \forall i \in \script{N} \label{eq:tslot-3} \\
& \alpha_i(\tau) \in \script{A}_i \: \: \forall \tau \in \{kT, \ldots, (k+1)T-1\} \label{eq:tslot-4} \\
& \gamma_i \in [0, u_i^{max}] \: \: \forall i \in \script{N} \label{eq:tslot-5}  
\end{eqnarray}
In particular, if $(\gamma_1^*, \ldots, \gamma_N^*)$ and $\{\alpha^*(\tau)\}_{\tau=kT}^{(k+1)T-1}$ forms an optimal solution to the above problem, then $\psi_T[k] = \phi(\gamma_1^*, \ldots, \gamma_N^*)$.  The above problem is always feasible with a well defined and finite supremum $\psi_T[k]$.  In general, the supremum may not be achievable.  However, for all $\epsilon>0$, there are vectors $(\gamma_1^*, \ldots, \gamma_N^*)$ and $\{\alpha^*(\tau)\}_{\tau=kT}^{(k+1)T-1}$ that satisfy the constraints of the above problem and that satisfy: 
\[ \psi_T[k] - \epsilon \leq \phi(\gamma_1^*, \ldots, \gamma_N^*) \leq \psi_T[k] \]

The value $\psi_T[k]$ represents the maximum average utility achievable over frame $k$ provided that the regret constraints are satisfied by averages over that frame, and assuming that future values of the vectors 
$b(t)$ and $\omega(t)$ are fully known over the frame.  Since $\psi_T[k]$ requires knowledge of the future to compute, it seems unlikely that a practical algorithm could achieve performance that is competitive with the $\psi_T[k]$ values.  Remarkably, 
this paper develops a \emph{causal} algorithm (without requiring knowledge of the future) that satisfies the 
constraints \eqref{eq:p2}-\eqref{eq:p3} and that achieves: 
\[ \liminf_{K\rightarrow\infty} \left[\overline{u}_i(KT) - \frac{1}{K}\sum_{k=0}^{K-1}\psi_T[k]\right] \geq -BT/V    \]
where $B$ is a constant and $V$ is a parameter that can be chosen as large as desired (with a  tradeoff in the corresponding convergence time).  
This holds under the same algorithm  
\emph{for all possible (finite) values of $T$}.    In particular, for any desired value of $T$, it is possible to choose a sufficiently large value of $V$ so that long term performance is arbitrarily close to the average of the $\psi_T[k]$ values. 

It can be shown that, for any $T$ and for any $\{\omega(t)\}_{t=0}^{\infty}$ and $\{b(t)\}_{t=0}^{\infty}$ sequences, the $\liminf$ average of $\psi_T[k]$ is less than or equal 
the optimal objective function value in \eqref{eq:p1}  for the infinite future knowledge optimization problem.  In cases when sequences $\{\omega(t)\}_{t=0}^{\infty}$ and $\{b(t)\}_{t=0}^{\infty}$ are ergodic and mild additional assumptions are satisfied,  the limiting average of $\psi_T[k]$ converges to this value 
as $T\rightarrow\infty$. 
However, this is not true in general non-ergodic situations. In particular, example sequences can be given for which the optimal value in \eqref{eq:p1} is strictly larger than the following value: 
\[ \liminf_{T\rightarrow\infty} \left[  \liminf_{K\rightarrow\infty} \frac{1}{K}\sum_{k=0}^{K-1}\psi_T[k]\right]  \]
Nevertheless, the average of $\psi_T[k]$  still provides a meaningful 
and challenging utility target. 

\section{Weighted sum of utilities} 

This subsection considers a weighted sum of utilities, so that: 
\[ \phi(u_1, \ldots, u_N) = \sum_{i=1}^N\theta_iu_i \]
for given real numbers $\theta_i$.    Subsection \ref{section:convex} considers the more general 
case when $\phi(\gamma_1, \ldots, \gamma_N)$ is any concave function. 

\subsection{Virtual queues} 

To achieve the time average constraint \eqref{eq:p2}, for each $i \in \script{N}$ define a \emph{virtual queue} $Q_i(t)$ that is initialized to $Q_i(0)=0$ and that has update equation: 
\begin{equation} \label{eq:q-update} 
Q_i(t+1) = \max[Q_i(t) + x_i(t) - u_i(t), 0] 
\end{equation} 
where $u_i(t)=\hat{u}_i(\alpha(t), \omega(t))$ and $x_i(t)=\hat{u}(\alpha(t), \omega(t))$.  These virtual queues are kept in the game manager and updated at the end of every round $t$ based on its suggestion $\alpha(t)$ and on 
knowledge of the $\omega(t)$ vector.  Recall that the manager knows all entries of $\omega(t)$ since each player $i$ tells it the values of 
$\omega_j(t)$ for $j \in \script{S}_i$, and $\cup_{i\in\script{N}} \script{S}_i = \{1, \ldots, M\}$. 

The following queueing lemma is standard for stochastic network optimization \cite{sno-text}. The proof is given for completeness. 

\begin{lem} \label{lem:virtual-queues}  (Virtual queues \cite{sno-text}) If $Q_i(t)$ satisfies \eqref{eq:q-update} then: 

a)  For all $t \in \{1, 2, 3, \ldots\}$ one has: 

\[ \overline{u}_i(t) \geq \overline{x}_i(t) - \left[\frac{Q_i(t)-Q_i(0)}{t}\right] \]

b) If $\lim_{t\rightarrow\infty} Q_i(t)/t=0$ for all $i \in \script{N}$, then constraints \eqref{eq:p2} are satisfied. 

\end{lem} 
\begin{proof} 
From \eqref{eq:q-update} one has for all $\tau \in \{0, 1, 2, \ldots\}$: 
\[ Q_i(\tau+1) \geq Q_i(\tau) + x_i(\tau) - u_i(\tau) \]
Thus: 
\[ Q_i(\tau+1) - Q_i(\tau) \geq x_i(\tau) - u_i(\tau) \]
Fix $t>0$. Summing the above over $\tau \in \{0, 1, \ldots, t-1\}$ gives: 
\[ Q_i(t)-Q_i(0) \geq \sum_{\tau=0}^{t-1} [x_i(\tau) - u_i(\tau)] \]
Dividing the result by $t$ and rearranging terms gives the result of part (a).  Part (b) immediately follows. 
\end{proof} 

A queue $Q_i(t)$ that satisfies $\lim_{t\rightarrow\infty} Q_i(t)/t = 0$ with probability 1 is said to be \emph{rate stable}.

\subsection{Drift-plus-penalty}

Define $Q(t) = (Q_1(t), \ldots, Q_N(t))$ and define $\norm{Q(t)}^2 = \sum_{i\in\script{N}} Q_i(t)^2$. 
Define $L(t) = \frac{1}{2}\norm{Q(t)}^2$, called a \emph{Lyapunov function}.  Define $\Delta(t) = L(t+1)-L(t)$.  The \emph{drift-plus-penalty} method of \cite{sno-text} observes $Q(t)$, $\omega(t)$, and $b(t)$ every round $t$, and then takes a control action $\alpha(t)$ to minimize a bound on the following expression: 
\[ \Delta(t) -V\sum_{i=1}^N\theta_iu_i(t) \]
where $V$ is a non-negative parameter that affects a performance tradeoff. 
The intuition is as follows:  Including the drift term $\Delta(t)$ in the above minimization maintains stable queues so that the time average constraints can be satisfied.  Including the penalty term $-V\sum_{i=1}^N\theta_iu_i(t)$ encourages the controller to make decisions that give a desirable weighted sum of utilities. 
Using larger values of $V$ places more emphasis on this ``penalty minimization.''  

\begin{lem} \label{lem:compute} Under any algorithm for choosing $\alpha(t) \in \script{A}_1\times \cdots \times \script{A}_N$, one has for all $t \in \{0, 1, 2, \ldots\}$: 
\begin{align} 
&\Delta(t) - V\sum_{i=1}^N\theta_iu_i(t) \nonumber \\
&\leq  B  - V\sum_{i=1}^N\theta_i\hat{u}_i(\alpha(t), \omega(t)) \nonumber \\
& + \sum_{i=1}^NQ_i(t)\left[\hat{u}_i(b(t),\omega(t)) - \hat{u}_i(\alpha(t),\omega(t))\right] \label{eq:DPP} 
\end{align} 
where $B = \frac{1}{2}\sum_{i=1}^N(u_i^{max})^2$. 
\end{lem} 
\begin{proof} 
Squaring \eqref{eq:q-update} and using $\max[y,0]^2 \leq y^2$ gives: 
\[ Q_i(t+1)^2 \leq Q_i(t)^2 + (x_i(t)-u_i(t))^2 + 2Q_i(t)[x_i(t)-u_i(t)] \]
Summing over $i \in \{1, \ldots, N\}$ and dividing by $2$ gives: 
\begin{align*} 
\Delta(t) &\leq \frac{1}{2}\sum_{i=1}^N(x_i(t)-u_i(t))^2 + \sum_{i=1}^NQ_i(t)[x_i(t)-u_i(t)] \\
&\leq B +  \sum_{i=1}^NQ_i(t)[x_i(t)-u_i(t)] 
\end{align*}
 where the last inequality follows because $(x_i(t)-u_i(t)) \in [-u_i^{max}, u_i^{max}]$ for all $t$. 
 Subtracting $V\sum_{i=1}^N\theta_iu_i(t)$ from both sides and using $u_i(t) = \hat{u}_i(\alpha(t), \omega(t))$ and $x_i(t) = \hat{u}_i(b(t), \omega(t))$  gives the result. 
\end{proof} 

The algorithm takes actions every round $t$ to greedily minimize the right-hand-side of the drift-plus-penalty expression \eqref{eq:DPP}.  The only terms on the right-hand-side of \eqref{eq:DPP}  that are affected by control decisions $\alpha(t)$ are: 
\[ -\sum_{i=1}^NV\theta_i\hat{u}_i(\alpha(t), \omega(t)) - \sum_{i=1}^NQ_i(t)\hat{u}_i(\alpha(t), \omega(t)) \]
The resulting algorithm is as follows:  Initialize $Q_i(0)=0$ for $i \in\script{N}$.  Every round $t \in \{0, 1, 2, \ldots\}$, the game manager observes vectors $Q(t)$, $\omega(t)$, and $b(t)$ and does the following: 

\begin{itemize} 
\item (Decisions) Choose suggestion vector $\alpha(t) = (\alpha_1(t), \ldots, \alpha_N(t))$ as the solution to the following: 
\begin{eqnarray*} 
\mbox{Maximize:} &
\sum_{i=1}^N\hat{u}_i(\alpha(t),\omega(t))[V\theta_i+ Q_i(t) ] \nonumber \\
\mbox{Subject to:} & \alpha_i(t) \in \script{A}_i \: \: \forall i \in \script{N}  \nonumber 
\end{eqnarray*} 

\item (Send suggestions) Send suggestions $\alpha_i(t)$ to each player $i \in \script{N}$. 

\item (Queue update) For each $i \in \script{N}$, update $Q_i(t)$ via \eqref{eq:q-update}. 
\end{itemize} 

It is clear that the above algorithm is causal:  It only requires knowledge of the current $Q(t), \omega(t), b(t)$ and  does not require knowledge of the future. 

\subsection{Performance for weighted utilities} 

The following theorem shows that the algorithm satisfies constraints \eqref{eq:p2}-\eqref{eq:p3} with constraint violations that decay like $O(\sqrt{V/t})$.

\begin{thm} \label{thm:constraints} (Constraint satisfaction) Fix $V\geq 0$ and assume the algorithm in the previous subsection is used with this $V$ and with $Q_i(0)=0$ for all $i\in\script{N}$. 
For all $t \in \{1, 2, 3, \ldots\}$ one has: 

a) $\frac{\norm{Q(t)}}{t} \leq \sqrt{\frac{2B + 2V\sum_{i=1}^N|\theta_i|u_i^{max}}{t}}$. 

b) $\overline{u}_i(t) - \overline{x}_i(t) \geq -\sqrt{\frac{2B + 2V\sum_{i=1}^N|\theta_i|u_i^{max}}{t}}$.

c) The constraints \eqref{eq:p2}-\eqref{eq:p3} are satisfied. 
\end{thm} 

\begin{proof} 
Fix $\tau \in \{0, 1,2 , \ldots\}$. 
Since the algorithm makes decisions for $\alpha(\tau)$ to minimize the right-hand-side of \eqref{eq:DPP}, one has: 
\begin{align} 
&\Delta(\tau) - V\sum_{i=1}^N\theta_iu_i(\tau) \nonumber \\
&\leq B - V\sum_{i=1}^N\theta_i\hat{u}_i(\alpha^*(\tau), \omega(\tau)) \nonumber \\
& + \sum_{i=1}^N Q_i(\tau)[\hat{u}_i(b(\tau), \omega(\tau)) - \hat{u}_i(\alpha^*(\tau), \omega(\tau))] \label{eq:plug} 
\end{align} 
where $\alpha^*(\tau)=(\alpha_1^*(\tau), \ldots, \alpha_N^*(\tau))$ is any alternative vector that
satisfies $\alpha_i^*(\tau) \in \script{A}_i$ for all $i \in \script{N}$.   A valid choice is $\alpha^*(\tau)=b(\tau)$. Substituting $\alpha^*(\tau)=b(\tau)$ into the right-hand-side of \eqref{eq:plug} gives:  
\[ \Delta(\tau) - V\sum_{i=1}^N\theta_iu_i(\tau) \leq B - V\sum_{i=1}^N\theta_i\hat{u}_i(b(\tau),\omega(\tau)) \]
Rearranging terms gives: 
\begin{equation} \label{eq:BVC} 
 \Delta(\tau) \leq B + VC 
 \end{equation} 
where $C = \sum_{i=1}^N|\theta_i|u_i^{max}$. The above holds for all rounds $\tau \in \{0, 1, 2, \ldots\}$.  Fix $t>0$.  Summing \eqref{eq:BVC} over $\tau\in \{0, \ldots, t-1\}$ gives: 
\[ L(t)-L(0) \leq (B+VC)t \]
Since $L(t) = \frac{1}{2}\norm{Q(t)}^2$ and $L(0)=0$ one has: 
\[ \norm{Q(t)}^2 \leq 2(B+VC)t \]
Dividing by $t^2$ and taking square roots gives: 
\[ \frac{\norm{Q(t)}}{t} \leq \sqrt{\frac{2(B+VC)}{t}} \]
This proves part (a).  Part (b) follows from (a) together with 
Lemma \ref{lem:virtual-queues}.  Part (c) follows directly from (b). 
 \end{proof} 

The above theorem does not use a value $T$.  The value $T$ is also never used in the algorithm implementation.  It is only used 
in the performance theorem below.  

\begin{thm}\label{thm:performance}  (Performance) Fix $V\geq 0$ and assume the algorithm in the previous subsection is used
with this $V$ and with $Q_i(0)=0$ for all $i \in \script{N}$.   For all positive integers $T$ and $K$ the following holds: 
\[ \sum_{i=1}^N \theta_i\overline{u}_i(KT) \geq \frac{1}{K}\sum_{k=0}^{K-1} \psi_T[k] - \frac{TB}{V} \]
\end{thm} 

Theorem \ref{thm:performance} shows that for any frame size $T$ and for any number of frames $K$, the weighted sum of utilities achieved by this algorithm over the first $T$ frames is at most $TB/V$ less than the average of the ideal $T$-slot lookahead values $\psi_T[k]$ over those frames.  Notice that this holds for all frame sizes $T$.  Since the algorithm does not use a value $T$ as input, the above theorem can be viewed as a class of performance bounds that are parameterized by $T$, all of which are satisfied.  
The error term 
$TB/V$ can be made as close to $0$ as desired by choosing $V$ appropriately large.  This is remarkable, particularly when $T$ is large, because the ideal $\psi_T[k]$ value is defined in terms of perfect knowledge of the future over $T$ rounds, whereas the algorithm does not know the future.  The tradeoff is that a large value of $V$ affects the \emph{convergence time} required to meet the desired constraints, as specified by part (b) of Theorem \ref{thm:constraints}. 

\begin{proof} (Theorem \ref{thm:performance}) 
Fix $k$ as a non-negative integer. Summing \eqref{eq:plug} over $\tau \in \{kT, \ldots, (k+1)T-1\}$ gives: 
\begin{align*} 
&\sum_{\tau=kT}^{(k+1)T-1}\Delta(\tau) - V\sum_{\tau=kT}^{(k+1)T-1}\sum_{i=1}^N\theta_iu_i(\tau) \\
&\leq TB - V\sum_{\tau=kT}^{(k+1)T-1}\sum_{i=1}^N\theta_i\hat{u}_i(\alpha^*(\tau), \omega(\tau)) \\
&+ \sum_{\tau=kT}^{(k+1)T-1}\sum_{i=1}^NQ_i(\tau)[\hat{u}_i(b(\tau),\omega(\tau)) - \hat{u}_i(\alpha^*(\tau), \omega(\tau))] \\
&\leq TB + 2B\sum_{m=0}^{T-1}m - V\sum_{\tau=kT}^{(k+1)T-1}\sum_{i=1}^N\theta_i\hat{u}_i(\alpha^*(\tau), \omega(\tau)) \\
&+ \sum_{\tau=kT}^{(k+1)T-1}\sum_{i=1}^NQ_i(kT)[\hat{u}_i(b(\tau),\omega(\tau)) - \hat{u}_i(\alpha^*(\tau), \omega(\tau))] \\
&= T^2B - VT\sum_{i=1}^N\theta_i\left[\frac{1}{T}\sum_{\tau=kT}^{(k+1)T-1}\hat{u}_i(\alpha^*(\tau), \omega(\tau)) \right] \\
&+ \sum_{\tau=kT}^{(k+1)T-1}\sum_{i=1}^NQ_i(kT)[\hat{u}_i(b(\tau),\omega(\tau)) - \hat{u}_i(\alpha^*(\tau), \omega(\tau))] \\
\end{align*} 
Now define $\gamma^*=(\gamma_1^*, \ldots, \gamma_N^*)$ by:  
\[ \gamma_i^* = \frac{1}{T}\sum_{\tau=kT}^{(k+1)T-1}\hat{u}_i(\alpha^*(\tau), \omega(\tau)) \: \: \forall i \in \script{N}  \]
Then: 
\begin{align*} 
&\sum_{\tau=kT}^{(k+1)T-1}\Delta(\tau) - V\sum_{\tau=kT}^{(k+1)T-1}\sum_{i=1}^N\theta_iu_i(\tau) \\
& \leq T^2B - VT\sum_{i=1}^N\theta_i\gamma_i^* \\
&+T\sum_{i=1}^NQ_k(kT)\left[\frac{1}{T}\sum_{\tau=kT}^{(k+1)T-1} \hat{u}_i(b(\tau),\omega(\tau)) - \gamma_i^*\right] 
\end{align*} 

Now fix $\epsilon>0$ and define $\alpha^*(\tau)$ for $\tau \in \{kT, \ldots, (k+1)T-1\}$ as decisions that 
satisfy the constraints \eqref{eq:tslot-2}-\eqref{eq:tslot-5} and yield: 
\[ \psi_T[k] - \epsilon \leq \sum_{i=1}^N \theta_i \gamma_i^* \leq \psi_T[k] \]
It follows that: 
\begin{align} 
&\sum_{\tau=kT}^{(k+1)T-1}\Delta(\tau) - V\sum_{\tau=kT}^{(k+1)T-1}\sum_{i=1}^N\theta_iu_i(\tau) \\
& \leq T^2B - VT\psi_T[k]  + \epsilon TV
\end{align} 
This holds for all $\epsilon>0$.  Taking a limit as $\epsilon\rightarrow 0$ gives: 
\begin{align*} 
&\sum_{\tau=kT}^{(k+1)T-1}\Delta(\tau) - V\sum_{\tau=kT}^{(k+1)T-1}\sum_{i=1}^N\theta_iu_i(\tau) \\
& \leq T^2B - VT\psi_T[k]  
\end{align*} 
Fix a positive integer $K$. 
Summing the above over $k \in \{0, 1, \ldots, K-1\}$ gives: 
\begin{align*}
&\sum_{\tau=0}^{KT-1}\Delta(\tau) - V\sum_{\tau=0}^{KT-1}\sum_{i=1}^N\theta_iu_i(\tau) \\
&\leq T^2BK - VT\sum_{k=0}^{K-1}\psi_T[k] 
\end{align*} 
That is: 
\begin{align*}
 &L(KT)  - L(0) -  VKT\sum_{i=1}^N\theta_i\overline{u}_i(KT) \\
 &\leq T^2BK - VKT\frac{1}{K}\sum_{k=0}^{K-1}\psi_T[k] 
 \end{align*}
Dividing by $VKT$ and using the fact that $L(0)=0$ and $L(KT)\geq 0$ gives: 
\[ - \sum_{i=1}^N\theta_i\overline{u}_i(KT) \leq \frac{TB}{V} - \frac{1}{K}\sum_{k=0}^{K-1}\psi_T[k] \]
Therefore: 
\[ \sum_{i=1}^N\theta_i\overline{u}_i(KT) \geq \frac{1}{K}\sum_{k=0}^{K-1} \psi_T[k] - \frac{TB}{V} \]
\end{proof} 

\subsection{Discussion} 

Theorems \ref{thm:constraints} and \ref{thm:performance} are \emph{deterministic results} that are guaranteed hold on every sample path, regardless of the probability model.  
Theorem \ref{thm:performance} holds for arbitrarily large values of $T$ and shows that, for large $K$, averages over $KT$ rounds give time average performance that is arbitrarily close to the value $\frac{1}{K}\sum_{k=0}^{K-1}\psi_T[k]$, where $\psi_T[k]$ is an ideal value based on knowledge of $T$ rounds into the future.  If the $\{b(t)\}_{t=0}^{\infty}$ and $\{\omega(t)\}_{t=0}^{\infty}$ processes are ergodic and mild additional assumptions are satisfied, then 
the value $\frac{1}{K}\sum_{k=0}^{K-1}\psi_T[k]$ approaches the ergodic supremum of the problem \eqref{eq:p1}-\eqref{eq:p3} as $T\rightarrow\infty$ and $K\rightarrow\infty$.  Intuitively, this is why the algorithm comes arbitrarily close to the solution of \eqref{eq:p1}-\eqref{eq:p3} in the ergodic case. 

However, the solution of  \eqref{eq:p1}-\eqref{eq:p3} requires perfect knowledge of \emph{all time into the future}, rather than just a finite horizon of $T$ slots into the future.  Thus, in the general 
non-ergodic case, the value 
$\frac{1}{K}\sum_{k=0}^{K-1}\psi_T[k]$ does not necessarily come close to the optimal objective function value in the problem
 \eqref{eq:p1}-\eqref{eq:p3}, regardless of how large the $T$ value is chosen to be.   Nevertheless, the value $\frac{1}{K}\sum_{k=0}^{K-1}\psi_T[k]$ is an intuitively ambitious target, and it is remarkable that, in all cases,  this target can be achieved (arbitrarily closely) by a causal algorithm.  
 

\section{Concave function of utilities} \label{section:convex} 

The previous section considered a weighted sum of utilities.  This section considers a 
general concave function of utilities.  Specifically, it uses a function $\phi(u_1, \ldots, u_N)$ that is concave and continuous over $(u_1, \ldots, u_N) \in \prod_{i=1}^N [0, u_i^{max}]$.  Thus, it treats the general problem  \eqref{eq:p1}-\eqref{eq:p3}.  For simplicity of exposition, assume the function $\phi(\cdot)$ is non-negative over $\prod_{i=1}^N [0, u_i^{max}]$ (else, a positive constant can be added to it to make it non-negative).

 The idea is to introduce \emph{proxy variables} $\gamma_i(t)$ that relate to running averages of player $i$ utility. The proxy variables are chosen every round $t$ in the interval $[0, u_i^{max}]$ and must satisfy: 
\[ \lim_{t\rightarrow\infty} \left[ \overline{\gamma}_i(t) - \overline{u}_i(t) \right] = 0 \: \: \forall i \in \script{N} \]
The above constraint is enforced by defining virtual queues $Z_i(t)$ for all $i \in \script{N}$ with update equation: 
\begin{equation} \label{eq:z-update} 
Z_i(t+1) = Z_i(t) + \gamma_i(t) - u_i(t) 
\end{equation} 
By summing \eqref{eq:z-update},  is clear that for all $t\in \{1, 2, 3, \ldots\}$ one has: 
\begin{equation} \label{eq:proxy-equality} 
 \overline{u}_i(t) = \overline{\gamma}_i(t) - \frac{Z_i(t)-Z_i(0)}{t} 
 \end{equation} 
 Therefore, it is desirable to make each queue $Z_i(t)$ rate stable.  The equality \eqref{eq:proxy-equality} 
implies: 
 \begin{equation} \label{eq:proxy-equality2} 
 \norm{\overline{u}(t) - \overline{\gamma}(t)} = \frac{\norm{Z(t)-Z(0)}}{t}
 \end{equation} 
where: 
\begin{eqnarray*}
\overline{u}(t) &=& (\overline{u}_1(t), \ldots, \overline{u}_N(t)) \\
\overline{\gamma}(t) &=& (\overline{\gamma}_1(t), \ldots, \overline{\gamma}_N(t))
\end{eqnarray*}

Let $Z(t) = (Z_1(t), \ldots, Z_N(t))$ be the vector of $Z_i(t)$ values. As before, define $Q(t)=(Q_1(t), \ldots, Q_N(t))$, with queues $Q_i(t)$ defined in  \eqref{eq:q-update}. Define a new Lyapunov function that considers both types of queues: 
\[ L(t) = \frac{1}{2}\norm{Q(t)}^2 + \frac{1}{2}\norm{Z(t)}^2\]
Define $\Delta(t) = L(t+1) - L(t)$.  The first step is to compute a bound on the new drift-plus-penalty expression: 
\[ \Delta(t) - V\phi(\gamma(t)) \]
where $\gamma(t) = (\gamma_1(t), \ldots, \gamma_N(t))$. 

\begin{lem} Under any algorithm for choosing $\alpha(t) \in \script{A}_1\times \cdots \script{A}_N$ and 
$\gamma(t) \in \prod_{i=1}^N[0, u_i^{max}]$, one has for all $t \in \{0, 1, 2, \ldots\}$: 
\begin{align}
&\Delta(t) - V\phi(\gamma(t)) \leq C - V\phi(\gamma(t)) \nonumber \\
& + \sum_{i=1}^N Q_i(t)[\hat{u}_i(b(t), \omega(t)) - \hat{u}_i(\alpha(t), \omega(t))] \nonumber \\
& + \sum_{i=1}^NZ_i(t)[\gamma_i(t)-\hat{u}_i(\alpha(t),\omega(t))] \label{eq:DPP2} 
\end{align}
where $C = \sum_{i=1}^N(u_i^{max})^2$. 
\end{lem} 

\begin{proof} 
The proof is similar to that of Lemma \ref{lem:compute} and is omitted for brevity. 
\end{proof} 

\subsection{General algorithm} 

The algorithm makes greedy decisions to minimize the right-hand-side of \eqref{eq:DPP2} on every round $t$. 
Specifically, every round $t \in \{0, 1, 2, \ldots\}$, the game manager observes $Q(t)$, $Z(t)$, $\omega(t)$, and $b(t)$ and does the following: 
\begin{itemize} 
\item (Proxy variables) Choose $\gamma(t)= (\gamma_1(t), \ldots, \gamma_N(t))$ as the solution to: 
\begin{eqnarray*}
\mbox{Maximize:} & V\phi(\gamma(t)) - \sum_{i=1}^NZ_i(t)\gamma_i(t)\\
\mbox{Subject to:} & 0 \leq \gamma_i(t) \leq u_i^{max} \: \: \forall i \in \script{N} 
\end{eqnarray*}

\item (Suggestions) Choose $\alpha(t) = (\alpha_1(t), \ldots, \alpha_N(t))$ as the solution to: 
\begin{eqnarray*}
\mbox{Maximize:} & \sum_{i=1}^N\hat{u}_i(\alpha(t),\omega(t))[Q_i(t)+Z_i(t)] \\
\mbox{Subject to:} & \alpha_i(t) \in \script{A}_i \: \: \forall i \in \script{N} 
\end{eqnarray*}
Next, for each $i \in \script{N}$, 
send suggestion $\alpha_i(t)$ to player $i$. 

\item (Queue update) For each $i \in \script{N}$, update $Q_i(t)$ and $Z_i(t)$ 
via \eqref{eq:q-update} and \eqref{eq:z-update}. 
\end{itemize} 

This is again a simple causal algorithm that is implemented as the game progresses. 

\subsection{Constraint analysis} 

Define $\phi^{max}$ as the maximum of $\phi(\gamma)$ over $\gamma \in \prod_{i=1}^N[0,u_i^{max}]$.  Such a maximum exists since it is defined over a compact subset of $\mathbb{R}^N$ and the function $\phi(\gamma)$ is continuous.  It is known that every continuous function over a compact set is \emph{Lipshitz continuous}, so that there is a positive value $M$ such that: 
\begin{equation} \label{eq:lipschitz} 
 |\phi(\gamma) - \phi(r)| \leq M\norm{\gamma - r} 
 \end{equation} 
for all $\gamma, r \in \prod_{i=1}^N[0,u_i^{max}]$. 

\begin{thm} \label{thm:constraint2} Fix $V\geq 0$ and assume the algorithm in the previous subsection 
is used with this $V$ and with $Q_i(0)=Z_i(0)=0$ for all $i \in \script{N}$.  For all 
$t \in \{1, 2, 3, \ldots\}$ one has: 

a) $\frac{\sqrt{\sum_{i=1}^N Q_i(t)^2 + Z_i(t)^2}}{t} \leq \sqrt{\frac{2C + 2V\phi^{max}}{t}}$

b) $\overline{u}_i(t) - \overline{x}_i(t) \geq -\sqrt{\frac{2C + 2V\phi^{max}}{t}}$. 

c) The constraints \eqref{eq:p2}-\eqref{eq:p3} are satisfied. 

d) The utilities satisfy: 
\[ \phi(\overline{u}(t)) \geq \frac{1}{t}\sum_{\tau=0}^{t-1}\phi(\gamma(\tau)) - M \sqrt{\frac{2C + 2V\phi^{max}}{t}} \]
where $\overline{u}(t) = (\overline{u}_1(t), \ldots, \overline{u}_N(t))$ and $M$ is the Lipschitz constant in 
\eqref{eq:lipschitz}. 
\end{thm} 

\begin{proof} Since the algorithm makes decisions $\gamma(t)$ and $\alpha(t)$ to minimize the right-hand-side of \eqref{eq:DPP2} every round, one has for all rounds $\tau \in \{0, 1, 2,\ldots\}$: 
\begin{align*}
&\Delta(\tau) - V\phi(\gamma(\tau)) \leq C - V\phi(\gamma^*) \\
& + \sum_{i=1}^NQ_i(\tau)[\hat{u}_i(b(\tau), \omega(\tau)) - \hat{u}_i(\alpha^*(\tau), \omega(\tau))]\\
&+\sum_{i=1}^NZ_i(\tau)[\gamma_i^* - \hat{u}_i(\alpha^*(\tau), \omega(\tau))] 
\end{align*}
for any alternative vectors $\gamma^* \in \prod_{i=1}^N[0, u_i^{max}]$ and $\alpha^*(\tau) \in \script{A}_1 \times \ldots\times  \script{A}_N$.  Define $\alpha^*(\tau)=b(\tau)$ and $\gamma^* = (\gamma_1^*, \ldots, \gamma_N^*)$ where $\gamma_i^* =\hat{u}_i(\alpha^*(\tau), \omega(\tau))$.  Then the above inequality becomes: 
\begin{align*}
&\Delta(\tau) - V\phi(\gamma(\tau)) \leq C - V\phi(\gamma^*) 
\end{align*}
Rearranging terms and using the fact that $0 \leq \phi(\gamma) \leq \phi^{max}$ for all $\gamma \in \prod_{i=1}^N[0, u_i^{max}]$ gives the following for all rounds $\tau$:  
\[ \Delta(\tau) \leq C + V\phi^{max} \]
The result of part (a) then follows by an argument similar to that of Theorem \ref{thm:constraints}. 

Part (b) follows from (a) together with \eqref{eq:proxy-equality}. 
 Part (c) follows immediately from part (b) by taking a limit. To prove part (d), note by Jensen's inequality for concave functions that for any round $t>0$: 
\begin{eqnarray*}
\frac{1}{t}\sum_{\tau=0}^{t-1} \phi(\gamma(\tau)) &\leq& \phi(\overline{\gamma}(t)) \\
&\leq& \phi(\overline{u}(t)) +  M\norm{\overline{\gamma}(t)-\overline{u}(t)} \\
&\leq& \phi(\overline{u}(t)) +  M\sqrt{\frac{2C + 2V\phi^{max}}{t}}
\end{eqnarray*}
where the second inequality above follows by \eqref{eq:lipschitz} and 
the final inequality follows by part (a) together with \eqref{eq:proxy-equality2}. 
\end{proof} 

Theorem \ref{thm:constraint2} provides a bound on the achieved performance
$\phi(\overline{u}(t))$ in terms of the time average $\frac{1}{t}\sum_{\tau=0}^{t-1}\phi(\gamma(\tau))$.  The next theorem completes the analysis by bounding the performance of this time average in terms of averages of the 
ideal $\psi_T[k]$ values over $T$-slot frames. 

\begin{thm} Fix $V \geq 0$ and assume the algorithm in the previous subsection is used with this $V$ and with $Q_i(0)=Z_i(0)=0$ for all $i \in \script{N}$.  For all positive integers $T$ and $K$ the following holds: 
\[ \frac{1}{KT}\sum_{\tau=0}^{KT-1} \phi(\gamma(\tau)) \geq \frac{1}{K}\sum_{k=0}^{K-1}\psi_T[k] - \frac{TC}{V} \]
\end{thm} 

\begin{proof} 
The proof is similar to that of Theorem \ref{thm:performance} and is omitted for brevity.
\end{proof} 

Combining the results of the above two theorems gives the following performance guarantee for all positive integers $T$ and $K$: 
\begin{eqnarray*}
\phi(\overline{u}(KT)) &\geq& \frac{1}{K}\sum_{k=0}^{K-1}\psi_T[k] \\
&&  - \frac{TC}{V} - M \sqrt{\frac{2C + 2V\phi^{max}}{KT}}
\end{eqnarray*}
Rearranging and taking a limit as $K\rightarrow\infty$ gives: 

\[ \liminf_{K\rightarrow\infty} \left[\phi(\overline{u}(KT)) - \frac{1}{K}\sum_{k=0}^{K-1}\psi_T[k]\right] \geq -\frac{TC}{V}\]
For any $T>0$, 
the right-hand-side above can be made arbitrarily small by an appropriately large value of $V$, with a corresponding convergence time tradeoff as specified in part (b) of Theorem \ref{thm:constraint2}. 

\section{More conservative constraints} 

This section considers a variation that requires the achieved utility of each player \emph{on each round} to be at least as large as the corresponding baseline utility for that round.  That is, the time average constraint \eqref{eq:p2} is replaced by the more restrictive constraint: 
\begin{equation}\label{eq:new} 
 u_i(t) \geq x_i(t) \: \: \forall i \in \script{N} \: , \: \forall t \in \{0, 1, 2, \ldots\} 
 \end{equation} 
where we recall that: 
\begin{eqnarray*}
u_i(t) &=& \hat{u}_i(\alpha(t), \omega(t)) \\
x_i(t) &=& \hat{u}_i(b(t), \omega(t)) 
\end{eqnarray*}
The resulting problem of interest is: 
\begin{align}
&\mbox{Maximize:} \nonumber \\
&\liminf_{t\rightarrow\infty} \phi(\overline{u}_1(t), \ldots, \overline{u}_N(t)) \label{eq:c1}  \\
&\mbox{Subject to:} \nonumber \\ 
&\hat{u}_i(\alpha(t), \omega(t)) \geq \hat{u}_i(b(t), \omega(t)) \forall i \in \script{N}, \forall t \in \{0, 1, 2, \ldots\} \label{eq:c2} \\
& \alpha_i(t) \in \script{A}_i \: \: \forall i \in \script{N}, \forall t \in \{0, 1, 2, \ldots \} \label{eq:c3} 
\end{align} 
where $\phi(u_1, \ldots, u_N)$ is again assumed to be a concave and continuous function over $\prod_{i=1}^N[0, u_i^{max}]$. 
The above problem is always feasible because the decisions $\alpha(t)=b(t)$ for all $t \in \{0, 1, 2, \ldots\}$ trivially satisfy the constraints \eqref{eq:c2}-\eqref{eq:c3}.   If the constraints of the above problem are satisfied, then the constraints of problem \eqref{eq:p1}-\eqref{eq:p3} are also satisfied  (but not vice versa).  Enforcing more restrictive constraints can reduce the optimal objective function value.  However, it provides players an immediate guarantee that the suggested decisions are at least as good as the baseline decisions, whereas problem \eqref{eq:p1}-\eqref{eq:p3} provides a similar guarantee only in the limit of a time average over multiple rounds.


The above problem can be written more simply as follows. 
Define: 
\begin{eqnarray*}
u(t) &=& (u_1(t), \ldots, u_N(t)) \\
\overline{u}(t) &=& (\overline{u}_1(t), \ldots, \overline{u}_N(t)) 
\end{eqnarray*}
For each $b \in \script{A}_1\times \cdots \times \script{A}_N$ and each $\omega \in \Omega$, define 
 $\script{A}(b,\omega)$ as the set of all action vectors $\alpha = (\alpha_1, \ldots, \alpha_N)$ that satisfy: 
\begin{eqnarray*}
\hat{u}_i(\alpha, \omega) \geq \hat{u}_i(b, \omega) & \forall i \in \script{N} \\
\alpha_i \in \script{A}_i & \forall i \in \script{N}
\end{eqnarray*}
The set $\script{A}(b,\omega)$ is non-empty for all $(b,\omega)$ because it contains the element $b$. The problem \eqref{eq:c1}-\eqref{eq:c3} is equivalent to the following: 
\begin{eqnarray} 
\mbox{Maximize:} & \liminf_{t\rightarrow\infty} \phi(\overline{u}(t)) \label{eq:s1} \\
\mbox{Subject to:} & \alpha(t) \in \script{A}(b(t),\omega(t)) \: \: \forall t \in \{0, 1, 2, \ldots\} \label{eq:s2}
\end{eqnarray} 

If the set $\script{A}(b(t),\omega(t))$ contains only the single element $b(t)$ for all $t \in \{0, 1, 2, \ldots\}$, then there are no decisions and the game manager is forced to choose $\alpha(t)=b(t)$ for all rounds $t$.   In this case, the problem is so restricted that the game manager cannot provide any utility gain. 
However, for many problems the sets $\script{A}(b(t),\omega(t))$ can have more than one element.

\subsection{Weighted sum of utilities} 

First consider the special case: 
\[ \phi(u_1, \ldots, u_N) = \sum_{i=1}^N\theta_i u_i \]
where $\theta_i$ are real numbers. 
In this special case, the following simple strategy is optimal: 
Every round $t \in \{0, 1, 2, \ldots\}$, the game manager observes $b(t)$ and $\omega(t)$ and then chooses a vector $\alpha(t) \in \script{A}(b(t),\omega(t))$ to maximize the following expression: 
\begin{equation} \label{eq:expression} 
  \sum_{i=1}^N\theta_i \hat{u}_i(\alpha(t), \omega(t)) 
  \end{equation} 

To see why this is optimal, consider any sequences $\{b(t)\}_{t=0}^{\infty}$ and $\{\omega(t)\}_{t=0}^{\infty}$. 
A sequence of actions $\{\alpha(t)\}_{t=0}^{\infty}$ is said to be a \emph{feasible sequence of actions} if $\alpha(t) \in \script{A}(b(t),\omega(t))$ for all $t$. 
Define $\{\alpha^*(t)\}_{t=0}^{\infty}$ as a feasible sequence of actions that maximize \eqref{eq:expression}
on every round $t$.  Let $\{\alpha'(t)\}_{t=0}^{\infty}$ be any alternative feasible sequence of actions.  
Then for all rounds $\tau\in \{0, 1, 2, \ldots\}$ one has: 
\[   \sum_{i=1}^N\theta_i \hat{u}_i(\alpha^*(\tau), \omega(\tau)) \geq \sum_{i=1}^N\theta_i\hat{u}_i(\alpha'(\tau), \omega(\tau)) \]
Fix $t$ as a positive integer. 
Summing the above over $\tau \in \{0, 1, 2, \ldots, t-1\}$ and dividing by $t$ proves that: 
\begin{equation} \label{eq:every-t} 
 \sum_{i=1}^N\theta_i \overline{u}^*_i(t) \geq \sum_{i=1}^N\theta_i\overline{u}_i'(t) 
 \end{equation} 
where $\overline{u}_i^*(t)$ is the time average utility of player $i$ over the first $t$ rounds under the actions 
$\{\alpha^*(\tau)\}_{\tau=0}^{\infty}$, while $\overline{u}_i'(t)$ is the corresponding time average 
utility under the actions $\{\alpha'(\tau)\}_{\tau=0}^{\infty}$.  The inequality \eqref{eq:every-t} is true for all rounds $t \in \{1, 2, 3, \ldots\}$ and so it is also true when taking limits as $t\rightarrow \infty$.

\subsection{General concave function of utilities} 

A naive attempt to solve \eqref{eq:s1}-\eqref{eq:s2} 
might consider the policy of observing $b(t)$ and $\omega(t)$ every round $t$ and then choosing $\alpha(t) \in \script{A}(b(t),\omega(t))$ to maximize $\phi(u_1(t), \ldots, u_N(t))$.  
The previous subsection shows this naive policy is optimal in the special case when $\phi(\cdot)$ is linear. However, it is not necessarily optimal when $\phi(\cdot)$ is concave but nonlinear.  

The problem \eqref{eq:s1}-\eqref{eq:s2} is similar to an \emph{opportunistic scheduling problem} for wireless networks, as considered in \cite{neely-fairness-ton}\cite{now}\cite{sno-text} via the drift-plus-penalty approach and in \cite{stolyar-greedy}\cite{atilla-fairness-ton}\cite{vijay-allerton02}\cite{prop-fair-down} via different approaches.  The works \cite{neely-fairness-ton}\cite{now}\cite{sno-text} transform a problem involving the maximization of a concave function of time averages into a problem of maximizing the time average of a function.  The same approach is fruitful in this game theory context.   Define proxy variables $\gamma_i(t)$ for all $i \in \script{N}$.  Define $\gamma(t) = (\gamma_1(t), \ldots, \gamma_N(t))$. Consider the following problem: 
\begin{eqnarray}
\mbox{Max:} & \liminf_{t\rightarrow\infty} \frac{1}{t}\sum_{\tau=0}^{t-1} \phi(\gamma(t)) \label{eq:t1} \\
\mbox{Subj to:} & \lim_{t\rightarrow\infty} [\overline{\gamma}_i(t) - \overline{u}_i(t)] = 0 \: \: \forall i \in \script{N}  \label{eq:t2} \\
& 0 \leq \gamma_i(t) \leq u_i^{max} \: \: \forall i \in \script{N}, \forall t \in \{0, 1, 2, \ldots\} \label{eq:t3} \\
& \alpha(t) \in \script{A}(b(t),\omega(t)) \: \: \forall t \in \{0, 1, 2, \ldots\} \label{eq:t4}  
\end{eqnarray}  

\begin{lem} \label{lem:one-side} If $\{\alpha(t)\}_{t=0}^{\infty}$ and $\{\gamma(t)\}_{t=0}^{\infty}$ are sequences of decisions that satisfy the constraints \eqref{eq:t2}-\eqref{eq:t4}, then $\{\alpha(t)\}_{t=0}^{\infty}$ is a feasible sequence of control actions for the original problem \eqref{eq:s1}-\eqref{eq:s2} with the following utility guarantee: 
\begin{equation} \label{eq:at-least} 
\liminf_{t\rightarrow\infty} \phi(\overline{u}(t)) \geq  \liminf_{t\rightarrow\infty} \frac{1}{t}\sum_{\tau=0}^{t-1}\phi(\gamma(\tau))  
\end{equation} 
\end{lem} 
\begin{proof} 
Let $\{\alpha(t)\}_{t=0}^{\infty}$ and $\{\gamma(t)\}_{t=0}^{\infty}$ be sequences
that satisfy  \eqref{eq:t2}-\eqref{eq:t4}. Fix $t>0$.  By Jensen's inequality and concavity of the $\phi(\gamma)$ function  one has: 
\begin{eqnarray}
 \frac{1}{t} \sum_{\tau=0}^{t-1}\phi(\gamma(\tau)) &\leq& \phi(\overline{\gamma}(t)) \nonumber \\
 &\leq& \phi(\overline{u}(t)) + M\norm{\overline{\gamma}(t) - \overline{u}(t)} \label{eq:lip} 
 \end{eqnarray}
 where $M$ is the Lipschitz constant for the function $\phi(\cdot)$.  Taking the $\liminf$ of both sides and using \eqref{eq:t2} proves \eqref{eq:at-least}.  Furthermore, since $\alpha(t)$ satisfies \eqref{eq:t4}, it is a feasible sequence of control actions for the original problem.
\end{proof} 

Lemma \ref{lem:one-side} suggests that one should make actions in an effort to solve the problem
\eqref{eq:t1}-\eqref{eq:t4}.   Any decisions that are feasible for the problem \eqref{eq:t1}-\eqref{eq:t4} and that produce a ``large'' value of the objective function will also be feasible for the original problem with a corresponding objective function value that is at least as large. Let $v_1^{opt}$ and $v_2^{opt}$ be the supremum objective function values for the problems \eqref{eq:s1}-\eqref{eq:s2} and \eqref{eq:t1}-\eqref{eq:t4}, respectively.  The above lemma implies that $v_1^{opt} \geq v_2^{opt}$.  If $(b(t), \omega(t))$ is ergodic, it turns out that $v_1^{opt} = v_2^{opt}$ with probability 1 (see \cite{sno-text}), although this is not necessarily true for general non-ergodic problems.  

The Lyapunov optimization method can be used to treat the problem \eqref{eq:t1}-\eqref{eq:t4}.  To enforce the constraints \eqref{eq:t2}, for each $i \in \script{N}$ define a virtual queue: 
\begin{equation} \label{eq:z2-update} 
Z_i(t+1) = Z_i(t) + \gamma_i(t) - u_i(t) 
\end{equation} 
Define $Z(t) = (Z_1(t), \ldots, Z_N(t))$. Define $L(t) = \frac{1}{2}\norm{Z(t)}^2$ and $\Delta(t) = L(t+1) - L(t)$. 
As before, it can be shown that: 
\begin{align}
&\Delta(t) - V\phi(\gamma(t)) \nonumber \\
&\leq D - V\phi(\gamma(t)) \nonumber \\
&  + \sum_{i=1}^NZ_i(t)[\gamma_i(t) - \hat{u}_i(\alpha(t), \omega(t))] \label{eq:conservative-DPP}
\end{align}
where $D$ is a constant. Minimizing the right-hand-side of \eqref{eq:conservative-DPP} every round $t$ results in the following algorithm:  Every round $t$, the game manager observes $Z(t)$, $b(t)$, $\omega(t)$.  Then: 
\begin{itemize} 
\item (Proxy variables)  Choose $\gamma(t) = (\gamma_1(t), \ldots, \gamma_N(t))$ as the solution to: 
\begin{eqnarray*}
\mbox{Maximize:} &  V\phi(\gamma(t)) - \sum_{i=1}^NZ_i(t)\gamma_i(t) \\
\mbox{Subject to:} & 0 \leq \gamma_i(t) \leq u_i^{max} \: \: \forall i \in \script{N}
\end{eqnarray*}
\item (Suggestions) Choose $\alpha(t) = (\alpha_1(t), \ldots, \alpha_N(t))$ as the solution to: 
\begin{eqnarray}
\mbox{Maximize:} & \sum_{i=1}^N Z_i(t) \hat{u}_i(\alpha(t), \omega(t)) \label{eq:sugg1} \\
\mbox{Subject to:} & \alpha(t) \in \script{A}(b(t),\omega(t)) \label{eq:sugg2} 
\end{eqnarray}
Then send these suggestions to the corresponding players. 
\item (Queue update) Update $Z_i(t)$ for $i \in \script{N}$ via \eqref{eq:z2-update}. 
\end{itemize} 

For simplicity, it is assumed throughout that for all $t\in\{0, 1,2,\ldots\}$ there exists an $\alpha(t)$ that solves \eqref{eq:sugg1}-\eqref{eq:sugg2} (else, the $C$-additive approximation theory of \cite{sno-text} can be used).  It follows that the resulting $\alpha(t)$ sequence is a feasible sequence of control actions for the original problem \eqref{eq:s1}-\eqref{eq:s2}.  The next subsection analyzes its  performance. 

\subsection{Analysis for conservative constraints and concave $\phi(\cdot)$} 

Fix sequences $\{\omega(t)\}_{t=0}^{\infty}$ and $\{b(t)\}_{t=0}^{\infty}$.  Fix a positive integer $T$. 
For $k \in \{0, 1, 2, \ldots\}$ consider the following $T$-slot lookahead problem, 
which uses decision variables $\gamma=(\gamma_1, \ldots, \gamma_N)$ and $\alpha(t)=(\alpha_1(t), \ldots, \alpha_N(t))$: 
\begin{align}
&\mbox{Maximize:} \nonumber \\
& \phi(\gamma)  \label{eq:conservative1} \\
&\mbox{Subject to:} \nonumber \\
& \gamma_i = \frac{1}{T}\sum_{\tau=kT}^{kT+T-1} \hat{u}_i(\alpha(\tau), \omega(\tau))\: \: \forall i \in \script{N} \label{eq:conservative2} \\
& \alpha(\tau) \in \script{A}(b(\tau), \omega(\tau)) \: \: \forall \tau \in \{kT, \ldots, (k+1)T-1\}\label{eq:conservative3} 
\end{align}

Define $\psi_T[k]$ as the supremum objective function value in the above problem. Thus, for any $\epsilon>0$, there is a sequence of decisions $\alpha(\tau)$ for $\tau \in \{kT, \ldots, (k+1)T-1\}$ and a vector $\gamma = (\gamma_1, \ldots, \gamma_N)$ that together satisfy \eqref{eq:conservative2}-\eqref{eq:conservative3} and also satisfy: 
\begin{equation} \label{eq:finally-satisfy} 
 \psi_T[k]-\epsilon \leq \phi(\gamma) \leq \psi_T[k]   
 \end{equation} 

\begin{thm} Fix $V\geq 0$ and assume the algorithm in the previous subsection is used with this $V$ and with $Z_i(0)=0$ for all $i \in \script{N}$. For all positive integers $T$ and $K$ the following holds: 
\[ \phi(\overline{u}(KT)) \geq \frac{1}{K}\sum_{k=0}^{K-1} \psi_T[k] - \frac{DT}{V} - M\sqrt{\frac{2D+2V\phi^{max}}{KT}} \]
where $M$ is the Lipschitz constant for the function $\phi(\cdot)$ and $D$ is the constant in \eqref{eq:conservative-DPP}. 
In particular, for all positive integers $T$ one has: 
\begin{equation} \label{eq:last} 
\liminf_{t\rightarrow\infty} \phi(\overline{u}(t)) \geq \liminf_{K\rightarrow\infty} \frac{1}{K}\sum_{k=0}^{K-1} \psi_T[k] - \frac{DT}{V} 
\end{equation} 
\end{thm} 

The theorem is proven in three parts. 

\begin{proof} (Part 1) 
This part proves that: 
\begin{equation} \label{eq:part1} 
\phi(\overline{u}(KT)) \geq \frac{1}{KT}\sum_{\tau=0}^{KT-1}\phi(\gamma(\tau)) - M\sqrt{\frac{2(D+V\phi^{max})}{KT}} 
\end{equation} 
To this end, note that \eqref{eq:conservative-DPP} implies that  for all rounds $\tau$: 
\begin{align} 
&\Delta(\tau) - V\phi(\gamma(\tau)) \nonumber \\
&\leq D - V\phi(\gamma^*(\tau)) \nonumber\\
&+\sum_{i=1}^NZ_i(\tau)[\gamma_i^*(\tau) - \hat{u}_i(\alpha^*(\tau),\omega(\tau))] \label{eq:have} 
\end{align} 
where $\gamma^*(\tau)$ and $\alpha^*(\tau)$ are any vectors that satisfy $\alpha(\tau) \in \script{A}(b(\tau),\omega(\tau))$ and $\gamma_i^*(\tau) \in [0, u_i^{max}]$ for all $i \in \script{N}$.
Choose $\alpha^*(\tau)=b(\tau)$ and $\gamma^*(\tau) = (\gamma_1^*(\tau), \ldots, \gamma_N^*(\tau))$ where $\gamma_i^*(\tau) = \hat{u}_i(b(\tau),\omega(\tau))$ for all $i \in \script{N}$.  Substituting these choices into \eqref{eq:have} gives: 
\[ \Delta(\tau) - V\phi(\gamma(\tau)) \leq D - V\phi(\gamma^*(\tau)) \] 
and hence: 
\[ \Delta(\tau) \leq D + V\phi^{max} \]
Summing over $\tau \in \{0, \ldots, KT-1\}$ gives: 
\[ \frac{1}{2}\norm{Z(KT)}^2 - \frac{1}{2}\norm{Z(0)}^2 \leq  KT(D+V\phi^{max}) \]
Rearranging terms and using $\norm{Z(0)}=0$ gives: 
\[ \frac{\norm{Z(KT)}}{(KT)} \leq \sqrt{\frac{2(D+V\phi^{max})}{KT}} \]
From \eqref{eq:z2-update} it holds that and $\norm{Z(KT)}/t = \norm{\overline{\gamma}(KT)-\overline{u}(KT)}$ and so: 
\[ \norm{\overline{\gamma}(KT)-\overline{u}(KT)} \leq \sqrt{\frac{2(D+V\phi^{max})}{KT}} \]
By \eqref{eq:lip} it follows that: 
\[ \frac{1}{KT}\sum_{\tau=0}^{KT-1} \phi(\gamma(\tau)) \leq \phi(\overline{u}(KT)) + M\sqrt{\frac{2(D+V\phi^{max})}{KT}} \]
which proves \eqref{eq:part1}. 
\end{proof} 

\begin{proof} (Part 2) This part shows that: 
\begin{equation} \label{eq:part2} 
\frac{1}{KT} \sum_{\tau=0}^{KT-1} \phi(\gamma(\tau))  \geq \frac{1}{K}\sum_{k=0}^{K-1} \psi_T[k] - \frac{DT}{V} 
\end{equation} 
To this end, note that summing \eqref{eq:have} over $\tau \in \{kT, \ldots, (k+1)T-1\}$ gives: 
\begin{align*}
&\sum_{\tau=kT}^{(k+1)T-1}\Delta(\tau) - V\sum_{\tau=kT}^{(k+1)T-1}\phi(\gamma(\tau))   \\
&\leq DT - V\sum_{\tau=kT}^{(k+1)T-1}\phi(\gamma^*(\tau)) \\
& + \sum_{\tau=kT}^{(k+1)T-1}\sum_{i=1}^NZ_i(\tau)[\gamma_i^*(\tau) - \hat{u}_i(\alpha^*(\tau), \omega(\tau))] \\
&\leq DT^2 - V\sum_{\tau=kT}^{(k+1)T-1}\phi(\gamma^*(\tau)) \\
& + \sum_{i=1}^NZ_i(kT)\sum_{\tau=kT}^{(k+1)T-1}[\gamma_i^*(\tau) - \hat{u}_i(\alpha^*(\tau), \omega(\tau))] \\
\end{align*}
where the final step is similar to a step in the proof of Theorem \ref{thm:performance}. 
Fix $\epsilon>0$ and 
define $\gamma^*=(\gamma_1^*, \ldots, \gamma_N^*)$ and $\alpha^*(\tau)$ for $\tau \in \{kT, \ldots, (k+1)T-1\}$ as the 
vectors that satisfy 
\eqref{eq:conservative2}, \eqref{eq:conservative3}, \eqref{eq:finally-satisfy}, and 
define $\gamma_i^*(\tau) = \gamma_i^*$ for all $\tau \in \{kT, \ldots, (k+1)T-1\}$.  Substituting these into the above inequality gives: 
\begin{align*}
&\sum_{\tau=kT}^{(k+1)T-1}\Delta(\tau) - V\sum_{\tau=kT}^{(k+1)T-1}\phi(\gamma(\tau))   \\
&\leq DT^2 - VT\psi_T[k] + VT\epsilon 
\end{align*}
Taking $\epsilon\rightarrow 0$ and then 
summing over $k \in \{0, \ldots, K-1\}$ gives: 
\[ L(KT)-L(0) - V\sum_{\tau=0}^{KT-1}\phi(\gamma(\tau)) \leq DT^2K - VT\sum_{k=0}^{K-1}\psi_T[k] \]
Dividing by $VKT$ and using the fact that $L(KT)-L(0)\geq 0$ gives: 
\[ -\frac{1}{KT}\sum_{\tau=0}^{KT-1} \phi(\gamma(\tau)) \leq \frac{DT}{V} - \frac{1}{K}\sum_{k=0}^{K-1} \psi_T[k] \]
which proves \eqref{eq:part2}.
\end{proof}

\begin{proof} (Part 3)  This part proves \eqref{eq:last}.  To this end, note that parts 1 and 2 together imply: 
\[ \phi(\overline{u}(KT)) \geq \frac{1}{K}\sum_{k=0}^{K-1} \psi_T[k] - \frac{DT}{V} - M\sqrt{\frac{2D+2V\phi^{max}}{KT}} \]
Taking a $\liminf$ of both sides as $K\rightarrow\infty$ gives: 
\begin{equation} \label{eq:batty} 
\liminf_{K\rightarrow\infty} \phi(\overline{u}(KT)) \geq \liminf_{K\rightarrow\infty} \frac{1}{K}\sum_{k=0}^{K-1} \psi_T[k] - \frac{DT}{V} 
\end{equation} 

It remains to show that the left-hand-side of \eqref{eq:batty} can be replaced by $\liminf_{t\rightarrow\infty} \phi(\overline{u}(t))$. To do this, 
fix $t$ as a positive integer.  Let $K_t$ be the non-negative integer such that $K_t T \leq t < (K_t+1)T$. Then: 
\[ \overline{u}(t) = \overline{u}(K_tT)\frac{K_tT}{t} + \frac{\sum_{\tau=K_tT}^{t-1}u(\tau)}{t}\]
In particular: 
\[ \overline{u}(t) = \overline{u}(K_tT) - \overline{u}(K_tT)\frac{(t-K_tT)}{t} + \frac{\sum_{\tau=K_tT}^{t-1}u(\tau)}{t} \]
Thus: 
\[ \norm{\overline{u}(t) - \overline{u}(K_tT)} \leq \frac{\sqrt{\sum_{i=1}^N(T-1)u_i^{max}}}{t} \]
Thus: 
\[ \phi(\overline{u}(t)) \geq \phi(\overline{u}(K_tT)) - M\frac{\sqrt{\sum_{i=1}^N(T-1)u_i^{max}}}{t} \]
Taking a $\liminf$ of both sides as $t\rightarrow\infty$ gives: 
\begin{eqnarray*}
 \liminf_{t\rightarrow\infty} \phi(\overline{u}(t)) &\geq& \liminf_{t\rightarrow\infty} \phi(\overline{u}(K_tT)) \\
 &\geq&  \liminf_{K\rightarrow\infty} \phi(\overline{u}(KT)) 
 \end{eqnarray*}
 This together with \eqref{eq:batty} proves the result. 
\end{proof} 

\section{Conclusion} 

This paper considers a stochastic repeated game where players share information and a baseline decision with a game manager at the beginning of each round.  The manager provides suggestions that, if taken, maximize a concave function of average utilities across players subject to the constraint that each player receives a time average utility at least as good as it would get if all players used their baseline strategies.  A more conservative scenario was also considered where the utility guarantee is enforced every round, rather than in a time average.  A Lyapunov optimization algorithm was developed that satisfies the constraints and that ensures the concave function of utilities is close to (or better than) an average of $T$-slot lookahead values that are computed with knowledge of $T$ rounds into the future, regardless of the sample path.   This shows that a simple causal algorithm can achieve a target defined in terms of future knowledge. 

\bibliographystyle{unsrt}
\bibliography{../../latex-mit/bibliography/refs}
\end{document}